\pdfoutput=1
\PassOptionsToPackage{unicode}{hyperref}
\PassOptionsToPackage{naturalnames}{hyperref}

\documentclass[11pt, reqno]{amsart}

\numberwithin{equation}{section}

\usepackage{amsmath,amsfonts,amssymb}
\usepackage{mathtools}
\usepackage[utf8]{inputenc}
\usepackage[english]{babel}
\usepackage{hyperref}
\usepackage{a4wide}
\usepackage{cite}
\usepackage{stmaryrd}
\usepackage{tikz-cd}

\DeclareUnicodeCharacter{FB01}{fi}
\newtheorem{theorem}{Theorem}[section]
\newtheorem{corollary}{Corollary}[theorem]
\newtheorem{lemma}[theorem]{Lemma}
\newtheorem{prop}[theorem]{Proposition}
\newtheorem*{remark}{Remark}
\newtheorem{examples}[theorem]{Examples}

\theoremstyle{definition}
\newtheorem{definition}{Definition}[section]

\begin{document}
	\title[Motives of parabolic Higgs bundles and parabolic connections]
	{On motives of parabolic Higgs bundles and parabolic connections}
	
	\author{Sumit Roy}
		\address{Center for Geometry and Physics, Institute for Basic Science (IBS), Pohang 37673, Korea}
\email{sumit.roy061@gmail.com}
\thanks{E-mail: sumit.roy061@gmail.com}
\thanks{Address: Center for Geometry and Physics, Institute for Basic Science (IBS), Pohang 37673, Korea}
	
	\begin{abstract}
	Let $X$ be a compact Riemann surface of genus $g \geq 2$ and let $D\subset X$ be a fixed finite subset. We considered the moduli spaces of parabolic Higgs bundles and of parabolic connections over $X$ with the parabolic structure over $D$. For generic weights, we showed that these two moduli spaces have equal Grothendieck motivic classes and their $E$-polynomials are the same. We also show that the Voevodsky and Chow motives of these two moduli spaces are also equal.  We showed that the Grothendieck motivic classes and the $E$-polynomials of parabolic Higgs moduli and of parabolic Hodge moduli are closely related. Finally, we considered the moduli spaces with fixed determinants and showed that the above results also hold for the fixed determinant case.
	\end{abstract}
	\maketitle
	
	\section{Introduction}
	In this paper we consider the moduli spaces of parabolic Higgs bundles (\cite{MS80}, \cite{H87a}, \cite{H87},\cite{S92},\cite{S94a}, \cite{BY96}) and parabolic connections (\cite{S94a}, \cite{N93}) over a compact Riemann surface $X$ of genus $g \geq 2$. These objects have several geometric structures that are useful in different topics, like algebraic geometry, differential geometry, mirror symmetry, mathematical physics, Langlands duality and so on. We prove equalities of some motivic classes (namely Grothendieck motives, Voevodsky motives and Chow motives) for these two moduli spaces. We also prove that their $E$-polynomials are equal.
	
	A \textit{parabolic bundle} $E_*$ over $X$ is a holomorphic vector bundle $E$ over $X$ together with a weighted flag over a fixed finite set $D\subset X$, called the \textit{parabolic structure}. These weights are real numbers between $0$ and $1$. A \textit{parabolic Higgs bundle} is a pair $(E_*,\Phi)$, where $E_*$ is a parabolic bundle and $\Phi : E_* \to E_* \otimes K(D)$ is a parabolic Higgs field, where $K$ is the canonical bundle over $X$. On the other hand, a \textit{parabolic connection} is a pair $(E_*,\mathcal{D})$ where $\mathcal{D} : E \to E \otimes K(D)$ is a logarithmic connection on the underlying vector bundle satisfying some properties. 
	
	In \cite{S94a}, Simpson considered an algebraic family over $\mathbb{C}$, which he called the Hodge moduli space, such that the fibers over $0$ and $1$ are exactly the moduli spaces of Higgs bundles and connections respectively. Also, that gives a homeomorphism between these two moduli spaces, which is famously known as the non-abelian Hodge correspondence (see \cite{S92}, \cite{S94a}, \cite{S95}). These two moduli spaces have singularities in general. But if we consider that the rank and degree are coprime, then these moduli spaces are smooth. For coprime rank and degree, Hausel and Thaddeus in \cite[Theorem 6.2]{HT03} proved that the $E$-polynomials are equal for these two moduli spaces and they have pure Hodge structure. There is a natural $\mathbb{C}^*$-action on the Hodge moduli space and with that, the moduli space becomes a semiprojective variety \cite{HV15}. Using the smoothness and semiprojectivity of the Hodge moduli space, Hoskins and Lehalleur in \cite{HL21} established a motivic version of the non-abelian Hodge correspondence. In that paper, they proved that the moduli of Higgs bundles and connections have equal motivic classes in various setups, see also \cite{GHS14} for motives of moduli spaces of Higgs bundles. Later, Federov, A. Soibelman, and Y. Soibelman in \cite{FSS20} proved a motivic Donaldson-Thomas invariants of these two moduli spaces in a parabolic setting.
	
	In this paper, we consider three types of motives, namely Grothendieck motives, Voevodsky motives and Chow motives. Let $\mathcal{V}_\mathbb{C}$ denote the category of complex quasi-projective varieties. Let $K(\mathcal{V}_\mathbb{C})$ denote the \textit{Grothendieck ring of varieties} and let $\hat{K}(\mathcal{V}_\mathbb{C})$ be the dimensional completion. Let $Z$ be a quasi-projective variety. Then $[Z] \in \hat{K}(\mathcal{V}_\mathbb{C})$ is called the \textit{motive} of $Z$. If $Z$ is $n$-dimensional with pure Hodge structure, then the corresponding $E$-polynomial is defined by 
	\[
	E(Z) = E(Z)(a,b) = \sum_{p,q=0}^{n} (-1)^{p+q}h_c^{p,q}(Z)u^pv^q.
	\]
Let $\mathcal{M}_{\mathrm{Higgs}}(r,d,\alpha)$, $\mathcal{M}_{\mathrm{pc}}(r,d,\alpha)$ and $\mathcal{M}_\mathrm{Hod}(r,d,\alpha)$ be the moduli space of parabolic Higgs bundles, moduli of parabolic connections and parabolic Hodge moduli space of rank $r$, degree $d$ and generic weights $\alpha$ over $X$ respectively. We proved the following Theorem
\begin{theorem}
In $\hat{K}(\mathcal{V}_\mathbb{C})$ we have the following motivic equalities
\[
[\mathcal{M}_\mathrm{Higgs}(r,d,\alpha)]=[\mathcal{M}_\mathrm{pc}(r,d,\alpha)] \hspace{0.2cm} \mathrm{and} \hspace{0.2cm} [\mathcal{M}_\mathrm{Hod}(r,d,\alpha)] = \mathbb{L}[\mathcal{M}_\mathrm{Higgs}(r,d,\alpha)].
\]
Therefore, we have the following equalities of the $E$-polynomials
\[
E(\mathcal{M}_\mathrm{Higgs}(r,d,\alpha))= E(\mathcal{M}_\mathrm{pc}(r,d,\alpha)) \hspace{0.2cm} \mathrm{and} \hspace{0.2cm} E(\mathcal{M}_\mathrm{Hod}(r,d,\alpha)) = uvE(\mathcal{M}_\mathrm{Higgs}(r,d,\alpha)).
\]
\end{theorem}
Here $\mathbb{L}$ is the Lefschetz motive. To prove this theorem, we first prove that the moduli spaces $\mathcal{M}_\mathrm{Higgs}(r,d,\alpha)$ and $\mathcal{M}_\mathrm{Hod}(r,d,\alpha)$ are semiprojective. For details of the proof see Theorem \ref{Grothendieck}. 

We then consider Voevodsky's category of geometric motives $DM_{\mathrm{gm}}(\mathbb{C}; R)$ over $\mathbb{C}$ with coefficients in a ring $R$. We denote by $M_{\mathrm{gm}}(X)_R$ the geometric motive of the smooth variety $X$ with coefficients in $R$. Also, we consider the category of Chow motives $\textbf{Chow}^{\mathrm{eff}}(\mathbb{C}; R)$ which has an embedding into the Voevodsky's category of effective geometric motives $DM^{\mathrm{eff}}_{\mathrm{gm}}(\mathbb{C}; R) \subset DM_{\mathrm{gm}}(\mathbb{C}; R)$. We denote by $C(X)_R$ the Chow motive of a smooth variety $X$ with coefficients in $R$. Then using the semi-projectivity, we prove the following motivic equalities
\begin{theorem}
For any ring $R$, we have the following isomorphisms of motives,
\[
M_{\mathrm{gm}}\big(\mathcal{M}_\mathrm{Higgs}(r,d,\alpha)\big)_R \cong M_{\mathrm{gm}}\big(\mathcal{M}_\mathrm{pc}(r,d,\alpha)\big)_R \in DM_{\mathrm{gm}}(\mathbb{C}; R).
\]
and 
\[
C\big(\mathcal{M}_\mathrm{Higgs}(r,d,\alpha)\big)_R \cong C\big(\mathcal{M}_\mathrm{pc}(r,d,\alpha)\big)_R \in \textbf{Chow}^{\mathrm{eff}}(\mathbb{C}; R).
\]
\end{theorem}
For details of the proof, see Theorem \ref{Voevodsky} and \ref{Chow}.

Finally, in the last section, we consider the moduli spaces with fixed determinants and we prove the above two theorems in the fixed determinant setup. See Theorem \ref{fixed1}, \ref{fixed2} and \ref{fixed3}.

\subsection{Future works}
There are some obvious questions raised by the outcomes we get. First, this paper derives properties of the moduli spaces of parabolic Higgs bundles and of parabolic connections, such as semiprojectivity using the rich geometry of the Hitchin fibration and we showed that the motives of these two different moduli spaces are the same. However, we believe that a closer inspection of these structures will yield many more geometrical properties of these moduli spaces. Since a motive of a variety should capture many of its cohomological invariants, we believe that we can derive several motivic properties of the moduli space of parabolic connections by studying the motives of parabolic Higgs bundles. We have explicitly discussed the motives of the moduli of parabolic Higgs bundles for rank $2$, odd degree case, and we expect that following arguments as in \cite{FHL22}, one can give an explicit description of the motives for rank $3$ case. 
	
	\section{Preliminaries}
	\subsection{Parabolic bundles}
	Let $X$ be a compact Riemann surface of genus $g \geq 2$ and let $D =\{p_1,\dots, p_n\} \subset X$ be a fixed subset of $n\geq 1$ distinct marked points of $X$.
		\begin{definition}\label{parabolic}
		 A \textit{parabolic bundle} $E_*$ of rank $r$ (assuming $r\geq 2$) over $X$ is a holomorphic vector bundle $E$ of rank $r$ over $X$ endowed with a parabolic structure along the divisor $D$, i.e. for every point $p \in D$, we have
	\begin{enumerate}
		\item a filtration of subspaces 
		\[
		E_p \eqqcolon E_{p,1}\supsetneq E_{p,2} \supsetneq \dots \supsetneq E_{p,r_p} \supsetneq E_{p,r_p+1} =\{0\},
		\]
		\item a sequence of real number satisfying 
		\[
		0\leq \alpha_1(p) < \alpha_2(p) < \dots < \alpha_{r_p}(p) < 1,
		\]
	\end{enumerate}
where $r_p$ is a natural number between $1$ and $r$. For all $i=1,\dots , r_p$, the real number $\alpha_i(p)$ is called the \textit{parabolic weight} associated to the subspace $E_{p,i}$. 
\end{definition}
For a fixed parabolic structure we denote the collection of all parabolic weights by $\alpha =\{(\alpha_1(p),\alpha_2(p),\dots,\alpha_{r_p}(p))\}_{p\in D}$. The parabolic structure is said to have \textit{full flags} if 
\[\mathrm{dim}(E_{p,i}/E_{p,i+1}) = 1
\] 
for all $i=1,\dots , r_p$ and for all $p\in D$, or equivalently $r_p=r$ for all $p\in D$.

 The \textit{parabolic degree} of a parabolic bundle $E_*$ is defined as
\[
\operatorname{pardeg}(E_*) \coloneqq \deg(E)+ \sum\limits_{p\in D}\sum\limits_{i=1}^{r_p} \alpha_i(p) \cdot \dim(E_{p,i}/E_{p,i+1})
\]
and the \textit{parabolic slope} of $E_*$ is defined as
\[
\mu_{\mathrm{par}}(E_*) \coloneqq \frac{\text{pardeg}(E_*)}{r}.
\]

In \cite{MY92}, Maruyama and Yokogawa gave an alternative definition of parabolic bundles in terms of coherent sheaves, which is useful to define the notion of parabolic tensor products and parabolic dual.

\begin{definition}
A \textit{parabolic homomorphism} $\phi : E_* \to E^\prime_*$ between two parabolic bundles is a homomorphism of underlying vector bundles that satisfies
the following: at each $p \in D$ we have 
\[
\alpha_i(p) > \alpha_j^\prime(p) \implies \phi(E_{p,i}) \subseteq E_{p,j+1}^\prime.
\]
Furthermore, we call such a homomorphism \textit{strongly parabolic} if
\[ \alpha_i(p) \geq \alpha_j^\prime(p) \implies \phi(E_{p,i}) \subseteq E_{p,j+1}^\prime
\]
for every $p \in D$.
\end{definition}

A parabolic subbundle $F_*\subset E_*$ is a holomorphic subbundle $F\subset E$ of the underlying vector bundle together with the induced parabolic structure, i.e. by taking the appropriate intersections.

A parabolic bundle $E_*$ is called \textit{stable} (resp. \textit{semistable}) if every nonzero proper subbundle $F_* \subset E_*$ satisfies
			\[
	\mu_{\mathrm{par}}(F_*) < \mu_{\mathrm{par}}(E_*) \hspace{0.2cm} (\mathrm{resp. } \hspace{0.2cm} \leq).
	\]
	
	The moduli space $\mathcal{M}(r,d,\alpha)$ of semistable parabolic bundles over $X$ of fixed rank $r$, degree $d$ and parabolic structure $\alpha$ was constructed by Mehta and Seshadri in \cite{MS80}. It is a normal projective complex variety of dimension
\[
\dim\mathcal{M}(r,d,\alpha) = r^2(g-1) + 1 + \dfrac{n(r^2-r)}{2}, 
\]
where the last summand comes from the assumption that the parabolic structure has full flags at each point $p\in D$. The stable locus of $\mathcal{M}(r,d,\alpha)$ is exactly the smooth locus of the moduli space. If weights are generic then semistability of a parabolic bundle implies stability, therefore the moduli space $\mathcal{M}(r,d,\alpha)$ is a smooth variety. 

\subsection{Parabolic Higgs bundles}
	Let $K$ be the canonical bundle on $X$. We write $K(D) \coloneqq K \otimes \mathcal{O}(D)$.
	
	\begin{definition}
	 A \textit{(strongly) parabolic Higgs bundle} on $X$ is a parabolic bundle $E_*$ on $X$ together with a Higgs field $\Phi : E_* \to E_* \otimes K(D)$ such that $\Phi$ is strongly parabolic, i.e. for all $p \in D$ we have $\Phi(E_{p,i}) \subset E_{p,i+1} \otimes \left.K(D)\right|_p$.
	\end{definition}
	
	We also have a notion of (non-strongly) parabolic Higgs bundle where the Higgs field $\Phi$ is a parabolic morphism. But in this paper, the Higgs field is always assumed to be strongly parabolic.
	
	For a parabolic Higgs bundle $(E_*,\Phi)$, a subbundle $F_*\subset E_*$ is called $\Phi$\textit{-invariant} if $\Phi$ preserves $F_*$, i.e. $\Phi(F_*)\subseteq F_* \otimes K(D)$.
	
	\begin{definition}
		A parabolic Higgs bundle $(E_*,\Phi)$ is said to be \textit{stable} (resp. \textit{semistable}) if every nonzero proper $\Phi$-invariant subbundle $F_* \subset E_*$ satisfies
			\[
	\mu_{\mathrm{par}}(F_*) < \mu_{\mathrm{par}}(E_*) \hspace{0.2cm} (\mathrm{resp. } \hspace{0.2cm} \leq).
	\]
\end{definition}

	Let $\mathcal{M}_{\mathrm{Higgs}}(r,d,\alpha)$ denote the moduli space of semistable parabolic Higgs bundles over $X$ of rank $r$, degree $d$ and full flag parabolic structure $\alpha$. It is a normal quasi-projective complex variety  (see \cite{Y93}, \cite{BY96}, \cite{T02}). The stable locus is exactly the smooth locus of the moduli space $\mathcal{M}_{\mathrm{Higgs}}(r,d,\alpha)$. If the weights are generic, then the semistability coincides with the stability of the moduli space. Thus, $\mathcal{M}_{\mathrm{Higgs}}(r,d,\alpha)$ is a smooth variety when the weights are generic.  
	
	Notice that the moduli space $\mathcal{M}(r,d,\alpha)$ is embedded in the moduli space $\mathcal{M}_{\mathrm{Higgs}}(r,d,\alpha)$ by considering the zero Higgs fields. By the para,sion of Serre duality (see \cite{Y93},\cite{Y95}), the cotangent bundle of $\mathcal{M}(r,d,\alpha)$
	\[
	T^*\mathcal{M}(r,d,\alpha) \subset \mathcal{M}_{\mathrm{Higgs}}(r,d,\alpha)
	\]
	is an open dense subset of $\mathcal{M}_{\mathrm{Higgs}}(r,d,\alpha)$. Thus the moduli space $\mathcal{M}_{\mathrm{Higgs}}(r,d,\alpha)$ has dimension
	\[
	\dim \mathcal{M}_{\mathrm{Higgs}}(r,d,\alpha) = 2\dim \mathcal{M}(r,d,\alpha) = 2r^2(g-1) + 2 + n(r^2-r).
	\]
	
	Let $t \in \mathbb{C}^*$ be a nonzero complex number. It can be checked that if $(E_*,\Phi) \in \mathcal{M}_{\mathrm{Higgs}}(r,d,\alpha)$ then so is $(E_*, t\Phi)$, i.e. $(E_*,t\Phi) \in \mathcal{M}_{\mathrm{Higgs}}(r,d,\alpha)$ is also a semistable parabolic Higgs bundle. Therefore, we have a standard $\mathbb{C}^*$-action on the moduli space $\mathcal{M}_{\mathrm{Higgs}}(r,d,\alpha)$, which is given by
	\[
	t \cdot (E_*,\Phi) = (E_*,t\Phi).
	\]

\subsubsection{Hitchin fibration}
Let $(E_*,\Phi)$ be a parabolic Higgs bundle of rank $r$. Consider the characteristic polynomial of the Higgs field $\Phi$,
	\[
	\det(x\cdot I - \Phi) = x^r + s_1x^{r-1} +\cdots +s_r,
	\]
	where $s_i =\mathrm{tr}(\wedge^i\Phi) \in H^0(X,K(D)^i)$ and $K(D)^i$ denotes the tensor product of $K^{\otimes i}$ and $i$-th power of the line bundle corresponding to the divisor $D$. 
	
	Since $\Phi$ is strongly parabolic, the residue of the parabolic Higgs field $\Phi$ at each marked point $p\in D$ is nilpotent with respect to the filtration. So the eigenvalues of $\Phi$ vanishes along the divisor $D$, i.e. $s_i \in H^0(X,K^i(D^{i-1})) \subset H^0(X,K(D)^i)$. Hence, we have the \textit{parabolic Hitchin fibration}
	\[
	h : \mathcal{M}_{\mathrm{Higgs}}(r,d,\alpha) \longrightarrow \mathcal{H} \coloneqq \bigoplus_{i=1}^{r} H^0(X, K^i(D^{i-1})),
	\]
	sending $(E_*,\Phi)$ to the coefficients of its characteristic polynomial of the Higgs field $\Phi$. Here the vector space $\mathcal{H}$ is called the \textit{Hitchin base}.
	
	Notice that the Hitchin fibration $h$ doesn't depend on the parabolic structure as it only depends on the Higgs field $\Phi$ and the line bundle $K(D)$. Also, $h$ is a proper surjective morphism (see \cite{M94}).
	
	Suppose $t \in \mathbb{C}^*$. Then the Hitchin base $\mathcal{H}$ also admits a natural $\mathbb{C}^*$-action, which is given by
	\begin{align*}
	    \mathbb{C}^* \times \mathcal{H} &\longrightarrow \mathcal{H}\\
	    (t, (s_1,s_2,\dots , s_r)) &\mapsto (ts_1,t^2s_2,\dots , t^rs_r).
	\end{align*}
	
	\subsection{Parabolic connections}
	A \textit{logarithmic connection} on a vector bundle $E$ over $X$, singular over the divisor $D$ is a holomorhic differential operator
	\[
	\mathcal{D} : E \to E \otimes K(D) 
	\]
	satisfying the Leibniz identity
	\begin{equation}\label{leibniz}
	    \mathcal{D}(fs)= f\mathcal{D}(s) + df\otimes s,
	\end{equation}
	where $f$ is a locally defined holomorphic function on $X$ and $s$ is a locally defined holomorphic section of $E$. For more details about the logarithmic connection, see \cite{D70} and \cite{BGKHME87}.

The fiber of $K(D)$ over a point $p\in D$ is identified with $\mathbb{C}$ by the Poincaré adjunction formula \cite[p. $146$]{GH78}. Consider the following morphism for a detailed explanation
 \begin{align*}
 \mathbb{C} &\longrightarrow \left.K(D)\right|_p \\ 
 x &\mapsto x\cdot \frac{dz}{z}(p),
 \end{align*}
 where $z$ is the coordinate function on $X$ such that $z(p)=0$ and defined on an open neighborhood of $p\in D$. It provides us with the necessary isomorphism. Let $E$ be a vector bundle that is holomorphic and has a logarithmic connection $\mathcal{D}$. We obtain the following $\mathcal{O}_X$-linear composition of morphisms from \eqref{leibniz}
 \[
E \xlongrightarrow[]{\mathcal{D}} E \otimes K(D) \xlongrightarrow[]{} (E \otimes K(D))_p \xlongrightarrow[]{\sim} E_p.
\]
The fact that $K(D)_p \cong \mathbb{C}$ implies the isomorphism $(E \otimes K(D))_p \xlongrightarrow[]{\sim} E_p$. Thus, we obtain a $\mathbb{C}$-linear morphism
\[
\mathrm{Res}(\mathcal{D},p) : E_p \longrightarrow E_p.
\]
We refer this $\mathbb{C}$-linear endomorphism as the \textit{residue} of $\mathcal{D}$ at $p\in D$ (see \cite{D70}). 

\begin{definition}
	A \textit{parabolic connection} on a parabolic bundle $E_*$ over $X$ is a logarithmic connection $\mathcal{D}$ on the underlying vector bundle $E$ satisfying the following conditions:
	\begin{enumerate}
	    \item On every fibre $E_p$ over each marked point $p \in D$, the logarithmic connection $\mathcal{D}$ satisfies 
	    \[
	    \mathcal{D}(E_{p,i}) \subseteq E_{p,i} \otimes \left.K(D)\right|_p
	    \]
	    for all $i = 1,2,\dots ,r$.
	    \item For all $p \in D$ and for all $i \in \{1,\dots , r\}$ the action of the residue $\mathrm{Res}(\mathcal{D},p) \in \mathrm{End}(E_p)$ on the quotient $E_{p,i}/E_{p,i+1}$ is the multiplication by $\alpha_{p,i}$, where $\alpha_{p,i}$'s are the parabolic weights over the point $p$.
	\end{enumerate}
\end{definition}	

\begin{remark}
    Since the residue $\mathrm{Res}(\mathcal{D},p)$ preserves the filtration over $p$, it acts on every quotient.
\end{remark}
 
	A parabolic connection will be denoted by $(E_*,\mathcal{D})$ and a parabolic subbundle $F_* \subset E_*$ is called $\mathcal{D}$\textit{-invariant} if $\mathcal{D}(F) \subseteq F \otimes K(D)$.
	
	\begin{definition}
	A parabolic connection $(E_*,\mathcal{D})$ is called \textit{stable} (resp. \textit{semistable}) if every non-zero proper $\mathcal{D}$-invariant subbundle $F_* \subset E_*$ satisfies
	\[
	\mu_{\mathrm{par}}(F_*) < \mu_{\mathrm{par}}(E_*) \hspace{0.3cm} (\mathrm{resp.}\hspace{0.05cm} \leq).
	\]
	\end{definition}
	The moduli space $\mathcal{M}_{\mathrm{pc}}(r,d,\alpha)$ of semistable parabolic connections of fixed rank $r$, degree $d$ and generic weight type $\alpha$ (assuming full flag structure) is a smooth quasi-projective irreducible complex variety of dimension 
	\[
	\dim\mathcal{M}_{\mathrm{pc}}(r,d,\alpha)= 2r^2(g-1) + 2 + n(r^2-r)
	\]
	(see \cite[Theorem 2.1]{I13}).
	
	\subsection{Parabolic $\lambda$-connections}
	Let $\lambda \in \mathbb{C}$ be a complex number. 
	\begin{definition}
	A \textit{parabolic $\lambda$-connection} over $X$ is a triple $(E_*,\lambda,\nabla)$ where $E_*$ is a parabolic bundle over $X$ and $\nabla : E \longrightarrow E \otimes K(D)$ is a holomorphic differential operator satisfying
	\begin{enumerate}
	    \item $\nabla(fs)= f\nabla(s) + \lambda \cdot df\otimes s$,
	where $f$ is a locally defined holomorphic function on $X$ and $s$ is a locally defined holomorphic section of $E$. 
	 \item On every fibre $E_p$, the connection $\nabla$ satisfies 
	    \[
	    \nabla(E_{p,i}) \subseteq E_{p,i} \otimes \left.K(D)\right|_p
	    \]
	    for all $i = 1,2,\dots ,r$.
	    \item For all $p \in D$ and for all $i \in \{1,\dots , r\}$ the action of the residue $\mathrm{Res}(\nabla,p)$ on the quotient $E_{p,i}/E_{p,i+1}$ is the multiplication by $\lambda\alpha_{p,i}$.
	\end{enumerate}
	\end{definition}
	A parabolic subbundle $F_* \subset E_*$ is called $\nabla$\textit{-invariant} if $\nabla(F) \subseteq F \otimes K(D)$. 
	\begin{definition}
	A parabolic $\lambda$-connection $(E_*,\lambda,\nabla)$ is \textit{stable} (resp. \textit{semistable}) if every non-zero proper $\nabla$-invariant subbundle $F_* \subset E_*$ satisfies
	\[
	\mu(F_*) < \mu(E_*) \hspace{0.3cm} (\mathrm{resp.}\hspace{0.05cm} \leq).
	\]
	\end{definition}
	We denote by $\mathcal{M}_\mathrm{Hod}(r,d,\alpha)$ the moduli space of semistable parabolic $\lambda$-connections over $X$ of fixed rank $r$, degree $d$ and weight type $\alpha$. For generic weights, the moduli space $\mathcal{M}_\mathrm{Hod}(r,d,\alpha)$ is a smooth quasiprojective complex variety (see \cite{A17}). This moduli space is also called the parabolic Hodge moduli space. 
	
	There is a canonical surjective algebraic map 
	\begin{equation}\label{submersion}
	    \mathrm{pr}\coloneqq \mathrm{pr}_\lambda : \mathcal{M}_\mathrm{Hod}(r,d,\alpha) \longrightarrow \mathbb{C}
	\end{equation}
	defined by $\mathrm{pr}(E_*,\lambda,\nabla)= \lambda$. 
	
	Let us consider the case $\lambda = 0$, i.e. the moduli space of parabolic $0$-connections $(E_*,0,\nabla)$. In this case the residue $\mathrm{Res}(\nabla,p)$ of the morphism $\nabla: E \longrightarrow E \otimes K(D)$ at every $p \in D$ acts as the zero map on the quotient $E_{p,i}/E_{p,i+1}$. Therefore for every $p\in D$, we have $\nabla(E_{p,i}) \subseteq E_{p,i+1} \otimes \left.K(D)\right|_p$. Thus, a parabolic $0$-connection is equivalent to a strongly parabolic Higgs bundle. Hence,
	\[
	\mathrm{pr}^{-1}(0) = \mathcal{M}_{\mathrm{Higgs}}(r,d,\alpha)  \subset \mathcal{M}_\mathrm{Hod}(r,d,\alpha). 
	\]
The natural $\mathbb{C}^*$-action on the moduli space $\mathcal{M}_{\mathrm{Higgs}}(r,d,\alpha)$ extends to a $\mathbb{C}^*$-action on $\mathcal{M}_\mathrm{Hod}(r,d,\alpha)$ defined by 
	\begin{equation}\label{action2}
	    	t\cdot (E_*,\lambda,\nabla) = (E_*,t\lambda,t\nabla).
	\end{equation}
	Similarly, if we consider the case $\lambda = 1$, then we get
	\[
	\mathrm{pr}^{-1}(1) = \mathcal{M}_{\mathrm{pc}}(r,d,\alpha) \subset \mathcal{M}_\mathrm{Hod}(r,d,\alpha).
	\]

\section{Semiprojectivity of the moduli spaces}	
In this section, we will prove the semiprojectivity of the moduli spaces $\mathcal{M}_\mathrm{Higgs}(r,d,\alpha)$ and $\mathcal{M}_\mathrm{Hod}(r,d,\alpha)$. 
	\begin{definition}[Semiprojective varieties]\label{semiprojective}
	Let $V$ be a quasi-projective complex variety equipped with a $\mathbb{C}^*$-action $z \mapsto t\cdot z$, $z\in V, t \in \mathbb{C}^*$. We say that $V$ is \textit{semiprojective} if it satisfies the following conditions:
	\begin{enumerate}
	    \item for every $x \in V$, the limit $\lim_{t\to 0} (t\cdot x)_{t\in \mathbb{C}^*}$ exists in $V$,
	    \vspace{0.1cm}
	    \item the fixed point locus $V^{\mathbb{C}^*}$ under the $\mathbb{C}^*$-action is proper.
\end{enumerate}
\end{definition}

\subsection{Semiprojectivity of the moduli space of parabolic Higgs bundles}\label{HiggsSemiproj}
\begin{lemma}\label{equivariant}
The Hitchin map $h : \mathcal{M}_{\mathrm{Higgs}}(r,d,\alpha) \to \mathcal{H}$ is $\mathbb{C}^*$-equivariant.
\end{lemma}
\begin{proof}
Recall that the $\mathbb{C}^*$-action on the Hitchin base $\mathcal{H}= \bigoplus_{i=1}^{r}H^0(X, K^i(D^{i-1}))$ is given by
\[
t\cdot (s_1,s_2,\dots , s_r) = (ts_1,t^2s_2,\dots , t^rs_r).
\]
Let $h\big((E_*,\Phi)\big) = (s_1,s_2,\dots , s_r)$, i.e. $s_i = \mathrm{tr}(\wedge^i\Phi)$ are the coefficients of the characteristic polynomial of $\Phi$. Then the characteristic polynomial of $t\Phi$ is given by
\[
\det(x\cdot I - t\Phi) = x^r + ts_1x^{r-1} + t^2s_2x^{r-2} + \cdots + t^rs_r.
\]
Therefore, 
\[
h\big((t\cdot (E_*,\Phi))\big)=h\big((E_*,t\Phi)\big)=(ts_1,t^2s_2,\dots , t^rs_r) = t\cdot (s_1,s_2,\dots , s_r) = t\cdot h\big((E_*,\Phi)\big).
\]
Hence, $h$ is $\mathbb{C}^*$-equivariant.
\end{proof}

To show the semiprojectivity of the moduli space $\mathcal{M}_{\mathrm{Higgs}}(r,d,\alpha)$ we have to show that the natural $\mathbb{C}^*$-action on $\mathcal{M}_{\mathrm{Higgs}}(r,d,\alpha)$ satisfies the above two conditions given in the Definition \ref{semiprojective}.
\begin{lemma}\label{limit1}
Let $(E_*,\Phi)$ be a semistable parabolic Higgs bundle. Then the limit $\lim_{t\to 0} (E_*,t\Phi)$ exists in $\mathcal{M}_{\mathrm{Higgs}}(r,d,\alpha)$.
\end{lemma}
\begin{proof}
Consider the map 
\[
f: \mathbb{C}^* \longrightarrow \mathcal{M}_{\mathrm{Higgs}}(r,d,\alpha)
\]
given by $t \mapsto (E_*,t\Phi)$. By the above Lemma \ref{equivariant}, we know that $h$ is $\mathbb{C}^*$-equivariant. Therefore, we have
\[
\lim_{t\to 0} h\big((E_*,t\Phi)\big) = \lim_{t \to 0} t \cdot h\big((E_*,\Phi)\big) = 0.
\]
Thus, the composition map $F \coloneqq h \circ f : \mathbb{C}^* \longrightarrow \mathcal{H}$ extends to a map $\hat{F} : \mathbb{C} \longrightarrow \mathcal{H}$. Since $h$ is proper, by the valuative criterion of properness $f$ also extend to a map $$\hat{f}: \mathbb{C} \longrightarrow \mathcal{M}_{\mathrm{Higgs}}(r,d,\alpha).$$ Hence, the limit $\lim_{t\to 0} (E_*,t\Phi)$ exists in $\mathcal{M}_{\mathrm{Higgs}}(r,d,\alpha)$.
\end{proof}

\begin{lemma}\label{fixed1}
The fixed point locus under the $\mathbb{C}^*$-action on $\mathcal{M}_{\mathrm{Higgs}}(r,d,\alpha)$ is proper in $h^{-1}(0) \subset \mathcal{M}_{\mathrm{Higgs}}(r,d,\alpha)$.
\end{lemma}
\begin{proof}
Note that the only element that is fixed under the $\mathbb{C}^*$-action on the Hitchin base $\mathcal{H}$ is the zero point. Therefore, the fixed point locus $\mathcal{H}^{\mathbb{C}^*}$ is exactly the set $\{0\}$. Since the Hitchin fibration $h$ is $\mathbb{C}^*$-equivariant, the fixed point locus $\mathcal{M}_{\mathrm{Higgs}}(r,d,\alpha)^{\mathbb{C}^*}$ must be closed in $h^{-1}(\mathcal{H}^{\mathbb{C}^*}) = h^{-1}(0)$. Since $h$ is proper, so is $h^{-1}(0)$. Hence, $\mathcal{M}_{\mathrm{Higgs}}(r,d,\alpha)^{\mathbb{C}^*}$ is also proper in $h^{-1}(0)$.
\end{proof}
\begin{prop}
The moduli space $\mathcal{M}_{\mathrm{Higgs}}(r,d,\alpha)$ is a smooth semiprojective complex variety.
\end{prop}
\begin{proof}
We know that for generic weights, the moduli space $\mathcal{M}_{\mathrm{Higgs}}(r,d,\alpha)$ is a smooth quasi-projective complex variety. Therefore the semiprojectivity of $\mathcal{M}_{\mathrm{Higgs}}(r,d,\alpha)$ follows from Lemma \ref{limit1} and \ref{fixed1}.
\end{proof}

\subsection{Semiprojectivity of parabolic Hodge moduli space}\label{HodgeSemiProj}
Recall that the $\mathbb{C}^*$-action on the moduli space $\mathcal{M}_\mathrm{Hod}(r,d,\alpha)$ is given by
\[
t\cdot (E_*,\lambda,\nabla) = (E_*,t\lambda,t\nabla).
\]
To prove the semiprojectivity of $\mathcal{M}_\mathrm{Hod}(r,d,\alpha)$ we need to check that this $\mathbb{C}^*$-action satisfies the two properties given in the Definition \ref{semiprojective}.

The following proposition is the statement of the Theorem $10.1$ in \cite{S97}. For the convenience of the reader, we give the full statement.

\begin{prop}\cite[Theorem $10.1$]{S97}
    Suppose $S=\mathrm{Spec}(A)$ where $A$ is a discrete valuation ring with fraction field $K$ and residue field $A/m=C$. Let $\eta$ denote the generic point and $s$ the closed point of $S$. Suppose $Y\to S$ is a projective morphism of schemes with relatively very ample $\mathcal{O}(1)$ on $Y$. Suppose $\Lambda$ is a split almost polynomial sheaf of rings of differential operators on $Y/S$ as in \cite{S94}. Suppose $F$ is a sheaf of $\Lambda$-modules on $Y$ which is relatively of pure dimension $d$, flat over $S$, and such that the generic fiber $F_\eta$ is semistable. Then there exists a sheaf of $\Lambda$-modules $F'$ on $Y$ which is relatively of pure dimension $d$, flat over $S$, and such that $F'_\eta \cong F_\eta$ and also $F'_s$ is semistable.
\end{prop}
 We will use this result to prove the following Lemma.
\begin{lemma}\label{limit2}
Let $(E_*,\lambda,\nabla) \in \mathcal{M}_\mathrm{Hod}(r,d,\alpha)$ be a semistable parabolic $\lambda$-connection. Then the limit $\lim_{t\to 0} (E_*,t\lambda,t\nabla)$ exists in $\mathrm{pr}^{-1}(0) \subset \mathcal{M}_\mathrm{Hod}(r,d,\alpha)$.
\end{lemma}
\begin{proof}
The proof is similar to \cite[Corollary 10.2]{S97}. Consider the following to projections
\[
\pi_1 : X \times \mathbb{C}^* \longrightarrow X \hspace{0.3cm} \mathrm{and} \hspace{0.3cm} \pi_2 : X \times \mathbb{C} \longrightarrow \mathbb{C}.
\]
Now consider the $\mathbb{C}^*$-flat family over $\pi_2: X \times \mathbb{C} \longrightarrow \mathbb{C}$ given by
\[
(\mathcal{E},t\lambda,\nabla_{\pi_2}) \coloneqq (\pi_1^*E_*,t\lambda,t\pi_1^*\nabla)
\]
For any $t\ne 0$, we know that a parabolic $t\lambda$-connection $(E_*,t\lambda,t\nabla)$ is semistable if and only if the parabolic $\lambda$-connection $(E_*,\lambda,\nabla)$ is semistable. Therefore, the fibers of the above family are semistable for $t\ne 0$. Following \cite[Theorem 10.1]{S97}, there exist a $\mathbb{C}$-flat family $(\overline{\mathcal{E}},\overline{t\lambda},\overline{\nabla_{\pi_2}})$ over $\pi_2 : X \times \mathbb{C} \longrightarrow \mathbb{C}$ such that 
\[
\left.(\overline{\mathcal{E}},\overline{t\lambda},\overline{\nabla_{\pi_2}})\right|_{X\times \mathbb{C}^*} \cong (\pi_1^*E_*,t\lambda,t\pi_1^*\nabla)
\]
and $\left.(\overline{\mathcal{E}},\overline{t\lambda},\overline{\nabla_{\pi_2}})\right|_{X\times \{0\}}$ is semistable. Therefore,
\[
\left.(\overline{\mathcal{E}},\overline{t\lambda},\overline{\nabla_{\pi_2}})\right|_{X\times \{0\}} \in \mathrm{pr}^{-1}(0)
\]
is the limit of the $\mathbb{C}^*$-orbit of $(E_*,\lambda,\nabla)$ at $t=0$ in the moduli space $\mathcal{M}_\mathrm{Hod}(r,d,\alpha)$.
\end{proof}

\begin{lemma}\label{fixed2}
The fixed point locus under the $\mathbb{C}^*$-action on the Hodge moduli $\mathcal{M}_{\mathrm{Hod}}(r,d,\alpha)$ is proper in $\mathcal{M}_{\mathrm{Hod}}(r,d,\alpha)$.
\end{lemma}
\begin{proof}
The fixed point locus under the $\mathbb{C}^*$-action $t\cdot (E_*,\lambda,\nabla) = (E_*,t\lambda,t\nabla)$ is exactly corresponds to the fixed point locus under the $\mathbb{C}^*$-action on $\mathrm{pr}^{-1}(0) = \mathcal{M}_{\mathrm{Higgs}}(r,d,\alpha)$. Therefore by Lemma \ref{fixed1}, the fixed point locus $\mathcal{M}_{\mathrm{Hod}}(r,d,\alpha)^{\mathbb{C}^*}$ is proper.
\end{proof}

We then generalize a result of Simpson \cite{S97}, see also Hausel-Thaddeus \cite{HT03}, Hausel-Rodriguez \cite{HR08}.

\begin{prop}\label{important}
The moduli space $\mathcal{M}_\mathrm{Hod}(r,d,\alpha)$ is a smooth semiprojective complex variety.

Moreover, the algebraic map $\mathrm{pr} : \mathcal{M}_{\mathrm{Hod}}(r,d,\alpha) \to \mathbb{C}$ given in (\ref{submersion}) is a $\mathbb{C}^*$-equivariant surjective submersion covering the scaling action on $\mathbb{C}$.
\end{prop}
\begin{proof}
Since weights are generic, the moduli space $\mathcal{M}_{\mathrm{Hod}}(r,d,\alpha)$ is a smooth quasiprojective complex variety. Therefore, Lemma \ref{limit2} and \ref{fixed2} implies that the moduli space $\mathcal{M}_{\mathrm{Hod}}(r,d,\alpha)$ is semiprojective.

The second part follows immediately from the smoothness property of the moduli space $\mathcal{M}_{\mathrm{Hod}}(r,d,\alpha)$.
\end{proof}

\section{Grothendieck motives of semiprojective varieties}
In this section we recall the Grothendieck ring of varieties and some basic properties. We also define what we mean by the Grothendieck motive.

\subsection{Grothendieck ring of varieties}
Let $\mathcal{V}_\mathbb{C}$ denote the category of quasiprojective complex varieties. We also denote by $[Z]$ the isomorphism class corresponding to an element $Z \in \mathcal{V}_\mathbb{C}$. Let $Z' \subset Z$ be a Zariski-closed subset of $Z$. Let $G$ be the quotient group coming from the free abelian group generated by the isomorphism classes $[Z]$, modulo the relation
\[
[Z] =[Z']+[Z\setminus Z'].
\]
In this group $G$, the additive structure is given by
\[
[Z_1]+[Z_2]\coloneqq [Z_1 \sqcup Z_2],
\]
where $\sqcup$ denotes the disjoint union and the multiplicative structure is defined by
\[
[Z_1]\cdot [Z_2]\coloneqq [Z_1 \times Z_2].
\]
Therefore we get a commutative ring $(G,+,\cdot)$, called the \textit{Grothendieck ring of varieties}. We will denote this ring by $K(\mathcal{V}_{\mathbb{C}})$. The additive and multiplicative units of $K(\mathcal{V}_{\mathbb{C}})$ are $0=[\emptyset]$ and $1=[\mathrm{Spec}(\mathbb{C})]$ respectively.

Consider the affine line $\mathbb{A}^1$. The class of $\mathbb{A}^1$ is called the \textit{Lefschetz object}, denoted by
\[
\mathbb{L} \coloneqq [\mathbb{A}^1]=[\mathbb{C}].
\]
Therefore, 
\[
\mathbb{L}^n = [\mathbb{A}^n]=[\mathbb{C}^n].
\]
Let $K(\mathcal{V}_{\mathbb{C}})[\mathbb{L}^{-1}]$ be the localization of $K(\mathcal{V}_{\mathbb{C}})$ and let 
\[
\hat{K}(\mathcal{V}_{\mathbb{C}})=\Bigg\{\sum_{k \geq 0}[Z_k]\mathbb{L}^{-k} \hspace{0.2cm} \Bigg| \hspace{0.2cm} [Z_k]\in K(\mathcal{V}_{\mathbb{C}}) \hspace{0.1cm} \mathrm{with} \hspace{0.1cm} \dim{Z_k} - k \longrightarrow -\infty \Bigg\}
\]
be the dimensional completion of $K(\mathcal{V}_{\mathbb{C}})$. 

Throughout this paper, by the Grothendieck motive we mean
\begin{definition}
let $Z$ be a quasiprojective complex variety. The class $[Z] \in K(\mathcal{V}_{\mathbb{C}})$ or in $\hat{K}(\mathcal{V}_{\mathbb{C}})$ is called the \textit{Grothendieck motive}, or just the \textit{motive} of $Z$.
\end{definition}

\subsection{Mixed Hodge structure and $E$-polynommial}
Let $d=\dim(Z)$ be the dimension of a quasiprojective complex variety $Z$. In \cite{D71}, Deligne proved that the compactly supported $k$-th cohomology $H^k_c(Z)\coloneqq H^k_c(Z,\mathbb{C})$ is equipped with a mixed Hodge structure for all $k\in \{0,\dots , 2d\}$. Also, $H^k_c(Z)$ is endowed with two filtrations $W^\bullet$ and $F_\bullet$ and that allow us to define the corresponding Hodge numbers
\[
h^{k,p,q}(Z) \coloneqq \dim H^{p,q}(H^k_c(Z,\mathbb{C})) = \dim\mathrm{Gr}^p_F\mathrm{Gr}^W_{p+q}(H^k_c(Z,\mathbb{C})),
\]
where $p,q \in \{0,\dots, k\}$. If $h^{k,p,q}(Z) \ne 0$, then we say that $(p,q)$ are $k$-weights of $Z$. It can be easily verify that the mixed Hodge numbers satisfy $h^{k,p,q}(Z)=h^{k,q,p}(Z)$ and $\dim H^k_c(Z) = \sum_{p,q=0}^{d} h^{k,p,q}(Z)$. Define
\[
\mathcal{X}^{p,q}(Z) \coloneqq \sum_k(-1)^kh^{k,p,q}(Z).
\]
Then the $E$\textit{-polynomial} of $Z$ is defined by
\[
E(Z)=E(Z;u,v) = \sum_{p,q=0} \mathcal{X}^{p,q}(Z)u^pv^q \in \mathbb{Z}[u,v].
\]
Notice that $E(Z;1,1)=\chi(Z)$ is the Euler characteristic of $Z$. So the $E$-polynomial is a generalization of the Euler characteristic. 

The $E$-polynomial satisfies the following properties
\begin{enumerate}
    \item \textit{(scissor relation)} $E(Z) = E(V) + E(Z\setminus V)$ for  a closed subvariety $V \subset Z$,
    \item \textit{(multiplicativity)} $E(Y \times Z) = E(Y)\cdot E(Z)$ where $Y \times Z$ is the cartesian product,
    \item If $Z \to Y$ is an algebraic fibre bundle with fibre $B$, then $E(Z)=E(Y)\cdot E(B)$.
\end{enumerate}

\begin{examples}
\begin{align*}
    &\bullet \hspace{0.1cm} E(\mathbb{C}) =E(\mathbb{A}^1)=E(\mathbb{P}^1) - E(\mathrm{pt})= uv \eqqcolon x,\\
    &\bullet \hspace{0.1cm} E(\mathbb{P}^n)=E(\mathbb{A}^n)+E(\mathbb{A}^{n-1})+\cdots + E(\mathbb{A}^1)+ E(\mathbb{A}^0)= x^n+x^{n-1}+\cdots + x + 1.
\end{align*}
\end{examples}
Now assume that $Z$ has pure Hodge structure, then its $E$-polynomial is given by
\begin{equation}\label{e-pol}
E(Z) = \sum_{p,q=0}^d (-1)^{p+q}h^{p,q}(Z)u^pv^q
\end{equation}
where $d=\dim Z$ and $h^{p,q}(Z)=\dim H^{p,q}_c(Z)$.

\begin{remark}\label{remark}
The $E$-polynomial can be realized as a ring homomorphism 
\[
E : \hat{K}(\mathcal{V}_{\mathbb{C}}) \longrightarrow \mathbb{Z}[u,v]
\]
from the Grothendieck ring of varieties to $\mathbb{Z}[u,v]$. This map extends to the completion 
\[
E : \hat{K}(\mathcal{V}_{\mathbb{C}}) \longrightarrow \mathbb{Z}[u,v] \Bigg\llbracket \frac{1}{uv} \Bigg\rrbracket
\]
(also denoted by $E$) taking values in the Laurent series in $uv$. Hence if two quasiprojective varieties have the same motive then their $E$-polynomials are the same.
\end{remark}

We will apply the following result for a smooth semiprojective complex variety in our set up.
\begin{prop}[\cite{AO21}, Theorem $5.6$]\label{mostimportant}
Let $Z$ be a smooth semiprojective complex variety endowed with a $\mathbb{C}^*$-equivarient surjective submersion $\pi : Z \to \mathbb{C}$ covering the standard scaling action on $\mathbb{C}$. Then the following motivic equalities hold in the Grothendieck ring $\hat{K}(\mathcal{V}_{\mathbb{C}})$,
\[
[\pi^{-1}(0)] = [\pi^{-1}(1)] \hspace{0.2cm} \mathrm{and} \hspace{0.2cm} [Z] = \mathbb{L}[\pi^{-1}(0)],
\]
where $\mathbb{L}$ is the Lefschetz motive.
\end{prop}
\begin{proof}
See \cite[Theorem $5.6$]{AO21} for details.
\end{proof}

\begin{theorem}\label{Grothendieck}
In $\hat{K}(\mathcal{V}_{\mathbb{C}})$ the following equalities hold,
\[
[\mathcal{M}_\mathrm{Higgs}(r,d,\alpha)]=[\mathcal{M}_\mathrm{pc}(r,d,\alpha)] \hspace{0.2cm} \mathrm{and} \hspace{0.2cm} [\mathcal{M}_\mathrm{Hod}(r,d,\alpha)] = \mathbb{L}[\mathcal{M}_\mathrm{Higgs}(r,d,\alpha)].
\]
Therefore, we have the following equalities of the $E$-polynomials
\[
E(\mathcal{M}_\mathrm{Higgs}(r,d,\alpha))= E(\mathcal{M}_\mathrm{pc}(r,d,\alpha)) \hspace{0.2cm} \mathrm{and} \hspace{0.2cm} E(\mathcal{M}_\mathrm{Hod}(r,d,\alpha)) = uvE(\mathcal{M}_\mathrm{Higgs}(r,d,\alpha)).
\]
\end{theorem}
\begin{proof}
By Proposition \ref{important}, the parabolic Hodge moduli space $\mathcal{M}_{\mathrm{Hod}}(r,d,\alpha)$ is a smooth semiprojective complex variety with the $\mathbb{C}^*$-action given in (\ref{action2}). Also, from the Proposition \ref{important} it follows that the surjective map $\mathrm{pr} : \mathcal{M}_{\mathrm{Hod}}(r,d,\alpha) \to \mathbb{C}$ given in (\ref{submersion}) is a $\mathbb{C}^*$-equivariant submersion covering the natural $\mathbb{C}^*$-action on $\mathbb{C}$. Therfore, by Proposition \ref{mostimportant}, we have
\[
[\mathcal{M}_\mathrm{Higgs}(r,d,\alpha)] = [\mathrm{pr}^{-1}(0)] = [\mathrm{pr}^{-1}(1)] = [\mathcal{M}_\mathrm{pc}(r,d,\alpha)]
\]
and 
\[
[\mathcal{M}_\mathrm{Hod}(r,d,\alpha)] = \mathbb{L}[\mathrm{pr}^{-1}(0)]=  \mathbb{L}[\mathcal{M}_\mathrm{Higgs}(r,d,\alpha)].
\]
Therefore, by Remark (\ref{remark}), $E$-polynomials of the moduli spaces $\mathcal{M}_\mathrm{Higgs}(r,d,\alpha))$ and $\mathcal{M}_\mathrm{pc}(r,d,\alpha))$ are equal, i.e. 
\[
E(\mathcal{M}_\mathrm{Higgs}(r,d,\alpha))= E(\mathcal{M}_\mathrm{pc}(r,d,\alpha))
\]
and by the multiplicative property of the $E$-polynomial, we have
\[
E(\mathcal{M}_\mathrm{Hod}(r,d,\alpha)) = E(\mathbb{C})E(\mathcal{M}_\mathrm{Higgs}(r,d,\alpha)) = uvE(\mathcal{M}_\mathrm{Higgs}(r,d,\alpha)).
\]
\end{proof}

\begin{corollary}
The Hodge structures of the moduli spaces $\mathcal{M}_\mathrm{Higgs}(r,d,\alpha)$ and $\mathcal{M}_\mathrm{pc}(r,d,\alpha)$ are isomorphic, i.e.
\[
H^{\bullet}(\mathcal{M}_\mathrm{Higgs}(r,d,\alpha)) = H^{\bullet}(\mathcal{M}_\mathrm{pc}(r,d,\alpha)).
\]
Also, the moduli spaces $\mathcal{M}_\mathrm{pc}(r,d,\alpha)$ and $\mathcal{M}_\mathrm{Hod}(r,d,\alpha)$ have pure mixed Hodge structures.
\end{corollary}
\begin{proof}
Following \cite[Corollary $1.3.3$]{HV15}, we have that the cohomologies of the fibres $\mathrm{pr}^{-1}(0) = \mathcal{M}_\mathrm{Higgs}(r,d,\alpha)$ and $\mathrm{pr}^{-1}(1) = \mathcal{M}_\mathrm{pc}(r,d,\alpha)$ are isomorphic and have pure mixed Hodge structures. Again by \cite[Corollary $1.3.2$]{HV15}, since the moduli space $\mathcal{M}_\mathrm{Hod}(r,d,\alpha)$ is smooth semiprojective for generic weights, it has pure cohomology.
\end{proof}

\section{Voevodsky motives and Chow motives}
In this section, we will briefly describe Voevodsky's category of geometric motives over a field $k$ with coefficients in a commutative ring $R$. This is a tensor triangulated category. For more details, see \cite{V00}, \cite{MVW06}, \cite{A04} and \cite{L08}.
 
 \subsection{The category of finite correspondences}
 \begin{definition}
 Let $Y$ and $Z$ be varieties over $k$. Let $c(Y,Z)$ denote the group generated by 
\[
\left\{\begin{tikzcd}[column sep=1em]
& W_i \arrow[dl,swap,"\pi_i"] \arrow[rd,"f_i"] \\
Y & & Z
\end{tikzcd}\right\}
\]
where
 \begin{enumerate}
 \item every $W_i \to Y \times Z$ is a closed immersion,
 \item every $W_i$ is integral,
\item every $\pi_i : W_i \longrightarrow Y$ is finite, and
\item every $\pi_i$ dominates an irreducible component of $Y$.
 \end{enumerate}
 Then the elements of the group $c(Y,Z)$ are called the \textit{finite correspondences} between the varieties $Y$ and $Z$.
 \end{definition}
 Let $X,Y$ and $Z$ be varieties over $k$, and let $W_1 \in c(X,Y)$ and $W_2 \in c(Y,Z)$ be two finite correspondences. If $X$ and $Y$ are irreducible then every irreducible component $P$ of $X \times |W_2| \cap |W_1| \times Z$ is finite and $\pi_X(P)=X$. Therefore, we have a bilinear composition rule
 \begin{align*}
     \circ : c(Y,Z) \times c(X,Y) &\longrightarrow c(X,Z)\\
     (W_2,W_1) &\mapsto W_2 \circ W_1 \coloneqq \pi_{(X\times Z)_*}\bigg(\pi^*_{(X\times Y)}(W_1)\cdot \pi^*_{(Y\times Z)}(W_2)\bigg).
 \end{align*}
 Consider the category of smooth $k$-varieties $\textbf{Sm}/k$. Then the objects of \textit{the category of finite correspondences} $\textbf{Corr}_{fin}/k$ are same as $\textbf{Sm}/k$, with
 \[
 \mathrm{Hom}_{\textbf{Corr}_{fin}/k}(Y,Z) \coloneqq c(Y,Z)
 \]
 and the composition law is given as above.
 \begin{remark}
 The operation $\times_k$ on $\textbf{Sm}/k$ and on cycles gives the category $\textbf{Corr}_{fin}/k$ the structure of a tensor category. Therefore, the corresponding bounded homotopy category $K^b(\textbf{Corr}_{fin}/k)$ is a tensor triangulated category.
 \end{remark}
 
 \subsection{The category of effective geometric motives}
 Consider the category $\widehat{DM}^{\mathrm{eff}}_{\mathrm{gm}}(k)$, which is the localization of the tensor triangulated category $K^b(\textbf{Corr}_{fin}/k)$, inverting all objects of the form $[X \times \mathbb{A}^1]\to [X]$ (homotopy) and $[U \cap V] \to [U]\oplus [V] \to [X]$ for any open covering $U \cup V = X$ (Mayer-Vietoris).
 \begin{definition}
 The category $DM^{\mathrm{eff}}_{\mathrm{gm}}(k)$ of \textit{effective geometric motives} over $k$ is the pseudo-abelian envelope of the quotient category $\widehat{DM}^{\mathrm{eff}}_{\mathrm{gm}}(k)$.
 \end{definition}
 
 We now consider the functor 
 \[
 \textbf{Sm}/k \longrightarrow \textbf{Corr}_{fin}/k
 \]
 sending a morphism $f : X \to Y$ in $\textbf{Sm}/k$ to its graph $\Gamma_f \subset X \times_k Y$. We will denote the object in $\textbf{Corr}_{fin}/k$ corresponding to $X \in \textbf{Sm}/k$ by $[X]$. This induces the following covariant functor
 \[
 M^{\mathrm{eff}}_{\mathrm{gm}} : \textbf{Sm}/k \longrightarrow DM^{\mathrm{eff}}_{\mathrm{gm}}(k)
 \]
 where $M^{\mathrm{eff}}_{\mathrm{gm}}(X)$ is the image of $[X]$ in $DM^{\mathrm{eff}}_{\mathrm{gm}}(k)$ and and it sends a morphism $f : X \to Y$ to $M^{\mathrm{eff}}_{\mathrm{gm}}(f) \coloneqq [\Gamma_f]$. 
 
  We note that the category $DM^{\mathrm{eff}}_{\mathrm{gm}}(k)$ is in fact a closed monoidal triangulated category. Therefore, we can consider the cone of a morphism, tensor products. The functor $M^{\mathrm{eff}}_{\mathrm{gm}}$ satisfies the following properties
 \begin{align*}
  M^{\mathrm{eff}}_{\mathrm{gm}}(X \sqcup Y) &= M^{\mathrm{eff}}_{\mathrm{gm}}(X) \oplus M^{\mathrm{eff}}_{\mathrm{gm}}(Y)\\
   M^{\mathrm{eff}}_{\mathrm{gm}}(X \times Y) &= M^{\mathrm{eff}}_{\mathrm{gm}}(X) \otimes M^{\mathrm{eff}}_{\mathrm{gm}}(Y).
 \end{align*}
 
 \begin{definition}
 $M^{\mathrm{eff}}_{\mathrm{gm}}(X)$ is said to be the \textit{effective geometric motive} of a smooth $k$-variety $X$.
 \end{definition}
 
 \subsubsection{Tate motives}
 Let $X \in \textbf{Sm}/k$ be a smooth variety with a $k$-point $0 \in X(k)$. Then the corresponding motive in $K^b(\textbf{Corr}_{fin}/k)$ is defined by
 \[
 \widehat{[X]} \coloneqq \mathrm{Cone}\bigg({i_0}_* : [\mathrm{Spec}(k)] \longrightarrow [X]\bigg).
 \]
 
 We denote the image of $ \widehat{[X]}$ in $M^{\mathrm{eff}}_{\mathrm{gm}}(k)$ by $\widehat{M^{\mathrm{eff}}_{\mathrm{gm}}(X)}$. We set
 \[
 \Lambda(1) \coloneqq \widehat{M^{\mathrm{eff}}_{\mathrm{gm}}(\mathbb{P}^1)}[-2].
 \]
 One can think $\Lambda(1)$ as the reduced homology of $\mathbb{P}^1$. It is an invertible object with respect to the tensor product and its inverse is exactly its dual $\underline{\mathrm{Hom}}(\Lambda(1),\Lambda(0))$. We denote its inverse by $\Lambda(-1)$.
 
 For $r \in \mathbb{Z}$, we set 
 \[
  \Lambda(r) =
  \begin{cases}
                                    \Lambda(1)^{\otimes r} & \text{if $r \geq 0$}\\
                                   \Lambda(-1)^{\otimes -r} & \text{if $r < 0$}.
  \end{cases}
\]
These objects are called \textit{pure Tate motives}. For an object $M \in DM^{\mathrm{eff}}_{\mathrm{gm}}(k)$, the twists
\[
M(r) \coloneqq M \otimes \Lambda(r)
\]
are called the \textit{Tate twists}.

\subsection{The category of geometric motives}
 To define the category of geometric motive we need to consider the motive $\Lambda(1)[2]$ which is similar to the Lefschetz motive $\mathbb{L}$.
 \begin{definition}
 The category $DM_{\mathrm{gm}}(k)$ of \textit{geometric motives} is defined by inverting the functor $\otimes_{\Lambda(1)}$ on $DM^{\mathrm{eff}}_{\mathrm{gm}}(k)$, i.e.  for $n,m \in \mathbb{Z}$ and $A,B \in DM^{\mathrm{eff}}_{\mathrm{gm}}(k)$,
 \[
 \mathrm{Hom}_{DM^{\mathrm{eff}}_{\mathrm{gm}}(k)}(A(n),B(m)) \coloneqq \lim_{{\longrightarrow}_{r}}\mathrm{Hom}_{DM^{\mathrm{eff}}_{\mathrm{gm}}(k)}\big(A\otimes \Lambda(r+n), B \otimes \Lambda(r+m)\big).
 \]
 \end{definition}
 The category of geometric motives $DM_{\mathrm{gm}}(k)$ is also a triangulated category. By Voevodsky's cancellation theorem, the embedding
 \[
 i : DM^{\mathrm{eff}}_{\mathrm{gm}}(k) \longrightarrow DM_{\mathrm{gm}}(k)
 \]
 is a fully faithful functor.
 
 Consider the composition
 \[
 M_{\mathrm{gm}} \coloneqq i \circ M^{\mathrm{eff}}_{\mathrm{gm}} : \textbf{Sm}/k \longrightarrow DM_{\mathrm{gm}}(k).
 \]
 \begin{definition}
 $M_{\mathrm{gm}}(X)$ is called the \textit{geometric motive} of the smooth $k$-variety $X$.
 \end{definition}
 
 Let $R$ be a ring and let $DM^{\mathrm{eff}}_{\mathrm{gm}}(k; R) \coloneqq DM^{\mathrm{eff}}_{\mathrm{gm}}(k) \otimes R$ denote the category of effective geometric motives with coefficients in $R$. We denote by $M^{\mathrm{eff}}_{\mathrm{gm}}(X)_R$ the effective geometric motive of $X$ in the category $DM^{\mathrm{eff}}_{\mathrm{gm}}(k; R)$. Similarly, we denote by $M_{\mathrm{gm}}(X)_R$ the geometric motive of $X$ in the category $DM_{\mathrm{gm}}(k; R) = DM_{\mathrm{gm}}(k) \otimes R$.
 
 \subsection{The category of effective Chow motives}
 Let $\textbf{Chow}^{\mathrm{eff}}(k; R)$ denote the category of effective Chow motives over a field $k$ with coefficients in $R$. There exist a functor
 \[
 \textbf{Chow}^{\mathrm{eff}}(k; R) \longrightarrow DM^{\mathrm{eff}}_{\mathrm{gm}}(k; R)
 \]
which is a fully faithful embedding. This functor is compatible with the tensor structure and the category $\textbf{Chow}^{\mathrm{eff}}(k; R)$ contains the Lefschetz motive $\mathbb{L}$. We can think of the category $DM^{\mathrm{eff}}_{\mathrm{gm}}(k; R)$ as being a "triangulated envelop" of the category $\textbf{Chow}^{\mathrm{eff}}(k; R)$. We can consider the motive of a smooth $k$-variety either in $\textbf{Chow}^{\mathrm{eff}}(k; R)$ or in the category $DM^{\mathrm{eff}}_{\mathrm{gm}}(k; R)$. See \cite{M68},\cite{S94} and \cite{B01} for more details.

Let $C(X)_R \in \textbf{Chow}^{\mathrm{eff}}(k; R)$ denote the \textit{Chow motive} of $X$ with coefficients in $R$.
 
\begin{theorem}\label{Voevodsky}
Let $X$ be a compact Riemann surface of genus $g \geq 2$. Then for any ring $R$, we have the following isomorphism of Voevodsky's motive,
\[
M_{\mathrm{gm}}\big(\mathcal{M}_\mathrm{Higgs}(r,d,\alpha)\big)_R \cong M_{\mathrm{gm}}\big(\mathcal{M}_\mathrm{pc}(r,d,\alpha)\big)_R \in DM_{\mathrm{gm}}(\mathbb{C}; R).
\]
\end{theorem}
\begin{proof}
We know that the moduli space $\mathcal{M}_\mathrm{Hod}(r,d,\alpha)$ is a smooth semiprojective variety equipped with a $\mathbb{C}^*$-invariant surjective submersion $\mathrm{pr} : \mathcal{M}_\mathrm{Hod}(r,d,\alpha) \to \mathbb{C}$ such that $\mathrm{pr}^{-1}(0) = \mathcal{M}_{\mathrm{Higgs}}(r,d,\alpha)$ and $\mathrm{pr}^{-1}(1) = \mathcal{M}_{\mathrm{pc}}(r,d,\alpha)$ (see \ref{submersion}). Therefore by \cite[Theorem B.1]{HL21}, we have the following isomorphism in the Voevodsky's category $DM_{\mathrm{gm}}(\mathbb{C}; R)$
\[
M_{\mathrm{gm}}\big(\mathcal{M}_\mathrm{Higgs}(r,d,\alpha)\big)_R = M_{\mathrm{gm}}\big(\mathrm{pr}^{-1}(0)\big)_R \cong M_{\mathrm{gm}}\big(\mathrm{pr}^{-1}(1)\big)_R = M_{\mathrm{gm}}\big(\mathcal{M}_\mathrm{pc}(r,d,\alpha)\big)_R.
\]
\end{proof}
This implies the following isomorphism of Chow motives.
\begin{theorem}\label{Chow}
For any ring $R$ we have the following isomorphism of Chow motives,
\[
C\big(\mathcal{M}_\mathrm{Higgs}(r,d,\alpha)\big)_R \cong C\big(\mathcal{M}_\mathrm{pc}(r,d,\alpha)\big)_R \in \textbf{Chow}^{\mathrm{eff}}(\mathbb{C}; R).
\]
\end{theorem}
\begin{proof}
Since the moduli space $\mathcal{M}_\mathrm{Higgs}(r,d,\alpha)$ is smooth semiprojective, its Voevodsky motive is pure by \cite[Corollary A.5]{HL21}. Similarly, the motive of the moduli space $\mathcal{M}_\mathrm{pc}(r,d,\alpha)$ is pure. Since their Voevodsky motives are isomorphic by the above Theorem \ref{Voevodsky}, their Chow motives are also isomorphic.
\end{proof}

\subsection{Rank two and odd degree case:} In \cite{FHL23}, the authors gave a closed formula for Chow motives of the moduli space of semistable parabolic Higgs bundles of rank $2$, odd degree $d$, and generic weight $\alpha$ with full flags at $n$ points $p_1, \dots, p_n \in X$. In \cite[Theorem $6.8$]{FHL23}, they showed that the (integral) Chow motives of $\mathcal{M}^{\alpha}_{\mathrm{Higgs}} = \mathcal{M}_\mathrm{Higgs}(2,d,\alpha)$ is given by the following isomorphism
\begin{align*}
    C(\mathcal{M}^\alpha_{\mathrm{Higgs}})_{\mathbb{Z}} \cong &C(\mathcal{M})_{\mathbb{Z}} \otimes C(\mathbb{P}^1)^{\otimes n}_{\mathbb{Z}}\\ &\oplus \bigoplus_{{0 \leq l \leq n} \atop {\tfrac{l+1-n}{2} \leq j \leq g-1}} C(\mathrm{Pic}^{a_{d,j}}(X))_{\mathbb{Z}} \otimes C(\mathrm{Sym}^{2j+n-l-1}(X))_{\mathbb{Z}}(3g-2j+l-2)^{\oplus \binom{n}{l}},
\end{align*}
where $a_{d,j} \coloneqq g-j + (d-1)/2$ and $\mathcal{M} = \mathcal{M}(2,d)$ is the moduli of semistable vector bundles over $X$. So, the (integral) Chow motives of $\mathcal{M}^{\alpha}_{\mathrm{Higgs}}$ are independent of weights and can be expressed in terms of the (integral) Chow motives of the non-parabolic case (see \cite{S02, BD07} for the description of motives of the moduli space of vector bundles over a curve).  Since the cohomology of $\mathcal{M}^{\alpha}_{\mathrm{Higgs}}$ is pure for generic weights, the Voevodsky motive is a Chow motive. Therefore, we get a description of the Voevodsky motive of $\mathcal{M}^{\alpha}_{\mathrm{Higgs}}$. Thus, from the above theorems \ref{Voevodsky} and \ref{Chow}, we get an explicit description of the (integral) Voevodsky (or, Chow) motives of the moduli space parabolic connections.

\section{Motives of moduli spaces with fixed determinant}
In earlier sections, we were working on the moduli space of parabolic bundles of fixed rank $r$ and degree $d$ over a curve $X$, which is the same as the moduli space of parabolic $\mathrm{GL}(r,\mathbb{C})$-bundles of degree $d$ over $X$. In this final section, we will consider the moduli space of parabolic $\mathrm{SL}(r,\mathbb{C})$-bundles over $X$, i.e. the moduli space of parabolic bundles over $X$ with fixed determinant.

By a \textit{parabolic $\mathrm{SL}(r,\mathbb{C})$-Higgs bundle} $(E_*,\Phi)$, we mean a parabolic bundle $E_*$ of rank $r$ with determinant $\xi$ and traceless Higgs field $\Phi$. Let $\mathrm{Jac}^d(X)$ denote the space of degree $d$ line bundles over $X$. Consider the determinant map
\begin{align*}
    \mathrm{det} : \mathcal{M}_{\mathrm{Higgs}}(r,d,\alpha) &\longrightarrow \mathrm{Jac}^d(X) \times H^0(X,K)\\
    (E_*,\Phi) &\longmapsto (\wedge^rE,\mathrm{trace}(\Phi)).
\end{align*}
Since $\Phi$ is strongly parabolic, $\mathrm{trace}(\Phi) \in H^0(X,K)\subset H^0(X,K(D))$. The moduli space $\mathcal{M}^{\xi}_{\mathrm{Higgs}}(r,d,\alpha)$ of semistable parabolic Higgs bundles with fixed determinant $\xi$ is defined by the fiber $\mathrm{det}^{-1}(\xi,0)$, i.e.
\[
\mathcal{M}^{\xi}_{\mathrm{Higgs}}(r,d,\alpha) \coloneqq \mathrm{det}^{-1}(\xi,0).
\]
As before, if weights are generic then the moduli space $\mathcal{M}^{\xi}_{\mathrm{Higgs}}(r,d,\alpha)$ is a smooth quasi-projective complex variety of dimension
\[
\dim\mathcal{M}^{\xi}_{\mathrm{Higgs}}(r,d,\alpha) = 2(g-1)(r^2-1)+n(r^2-r).
\]
In this case, as $\mathrm{trace}(\Phi)=0$, the Hitchin map is given by 
\[
	h^{\xi} : \mathcal{M}^{\xi}_{\mathrm{Higgs}}(r,d,\alpha) \longrightarrow \mathcal{H}^{\xi} \coloneqq \bigoplus_{i=2}^{r} H^0(X, K^i(D^{i-1})).
\]
Following \cite[theorem $1.2$]{BGL11} and (\ref{HiggsSemiproj}), similarly we can prove that the Hitchin map $h^{\xi}$ is $\mathbb{C}^*$-equivariant and the moduli space $\mathcal{M}^{\xi}_{\mathrm{Higgs}}(r,d,\alpha)$ is smooth $\mathbb{C}^*$-equivariant closed subvariety of $\mathcal{M}_{\mathrm{Higgs}}(r,d,\alpha)$. Therefore, $\mathcal{M}^{\xi}_{\mathrm{Higgs}}(r,d,\alpha)$ is smooth semiprojective.

By a \textit{parabolic $\lambda$-connection with fixed determinant $(\xi,\delta)$} (i.e. for the group $\mathrm{SL}(r,\mathbb{C})$), we mean a parabolic $\lambda$-connection  $(E_*,\lambda,\nabla)$ such that $\wedge^r E_* \cong \xi$ and $\mathrm{trace}(\nabla) =\delta$ (see \cite[Definition $8.1$]{A17} for more details). It can be verified that the $\mathrm{trace}(\nabla)$ gives a $\lambda$-connection on the line bundle $\xi$. Consider the determinant map
\begin{align*}
    \mathrm{det} : \mathcal{M}_{\mathrm{Hod}}(r,d,\alpha) &\longrightarrow \mathcal{M}_{\mathrm{Hod}}(1,d,\alpha)\\
    (E_*,\lambda,\nabla) &\longmapsto (\wedge^rE_*,\lambda,\mathrm{trace}(\nabla)).
\end{align*}

Then the moduli space $\mathcal{M}^{\xi}_{\mathrm{Hod}}(r,d,\alpha)$ of semistable parabolic $\lambda$-connections with fixed determinant $(\xi,\delta)$ is defined by
\[
\mathcal{M}^{\xi}_{\mathrm{Hod}}(r,d,\alpha) \coloneqq \mathrm{det}^{-1}(\xi,\lambda, \delta).
\]
The moduli space $\mathcal{M}^{\xi}_{\mathrm{Hod}}(r,d,\alpha)$ is clearly a smooth $\mathbb{C}^*$-invariant closed subvariety of $\mathcal{M}_{\mathrm{Hod}}(r,d,\alpha)$. Therefore, following (\ref{HodgeSemiProj}), we can similarly prove that $\mathcal{M}^{\xi}_{\mathrm{Hod}}(r,d,\alpha)$ is in fact semiprojective. By considering the restriction of the morphism (\ref{submersion}), we get a $\mathbb{C}^*$-invariant surjective submersion
\[
\mathrm{pr} : \mathcal{M}^{\xi}_{\mathrm{Hod}}(r,d,\alpha) \longrightarrow \mathbb{C}.
\]
Let $\mathcal{M}^{\xi}_{\mathrm{pc}}(r,d,\alpha)$ denote the moduli space of semistable parabolic connections with fixed determinant $(\xi,\delta)$.

Then we have the following isomorphisms 
\begin{enumerate}
    \item $\mathrm{pr}^{-1}(0) \cong \mathcal{M}^{\xi}_{\mathrm{Higgs}}(r,d,\alpha)$
    \item $\mathrm{pr}^{-1}(1) \cong \mathcal{M}^{\xi}_{\mathrm{pc}}(r,d,\alpha)$
    \item $\mathrm{pr}^{-1}(\mathbb{C}^*) \cong \mathcal{M}^{\xi}_{\mathrm{pc}}(r,d,\alpha) \times \mathbb{C}^*$.
\end{enumerate}

Then we have the following motivic invariance theorems.
\begin{theorem}[Grothendieck motive]\label{fix1}
In the Grothendieck ring of varieties $\hat{K}(\mathcal{V}_{\mathbb{C}})$ the following equalities hold,
\[
[\mathcal{M}^{\xi}_\mathrm{Higgs}(r,d,\alpha)]=[\mathcal{M}^{\xi}_\mathrm{pc}(r,d,\alpha)] \hspace{0.2cm} \mathrm{and} \hspace{0.2cm} [\mathcal{M}^{\xi}_\mathrm{Hod}(r,d,\alpha)] = \mathbb{L}[\mathcal{M}^{\xi}_\mathrm{Higgs}(r,d,\alpha)].
\]
Therefore, we have the following equalities of the $E$-polynomials
\[
E(\mathcal{M}^{\xi}_\mathrm{Higgs}(r,d,\alpha))= E(\mathcal{M}^{\xi}_\mathrm{pc}(r,d,\alpha)) \hspace{0.2cm} \mathrm{and} \hspace{0.2cm} E(\mathcal{M}^{\xi}_\mathrm{Hod}(r,d,\alpha)) = uvE(\mathcal{M}^{\xi}_\mathrm{Higgs}(r,d,\alpha)).
\]
\end{theorem}
\begin{proof}
The proof is analogous to the proof of the Theorem \ref{Grothendieck}. The objects simply need to be carefully changed to the fixed determinant version.
\end{proof}

\begin{theorem}[Voevodsky motive]\label{fix2}
For any ring $R$, we have the following isomorphism of Voevodsky's motive,
\[
M_{\mathrm{gm}}\big(\mathcal{M}^{\xi}_\mathrm{Higgs}(r,d,\alpha)\big)_R \cong M_{\mathrm{gm}}\big(\mathcal{M}^{\xi}_\mathrm{pc}(r,d,\alpha)\big)_R \in DM^{\mathrm{eff}}_{\mathrm{gm}}(\mathbb{C}; R).
\]
\end{theorem}
\begin{proof}
The proof is the same as in Theorem \ref{Voevodsky}.
\end{proof}

\begin{theorem}[Chow motive]\label{fix3}
For any ring $R$ we have the following isomorphism of Chow motives,
\[
C\big(\mathcal{M}^{\xi}_\mathrm{Higgs}(r,d,\alpha)\big)_R \cong C\big(\mathcal{M}^{\xi}_\mathrm{pc}(r,d,\alpha)\big)_R \in \textbf{Chow}^{\mathrm{eff}}(\mathbb{C}; R).
\]
\end{theorem}
\begin{proof}
The proof is the same as in Theorem \ref{Chow}.
\end{proof}

In \cite[Theorem $6.9$]{FHL23}, the authors showed that the (integral) Chow motives of the moduli space $\mathcal{M}^{\alpha}_{\mathrm{Higgs}}(\xi) = \mathcal{M}^{\xi}_\mathrm{Higgs}(2,d,\alpha)$ of semistable parabolic Higgs bundles of rank $2$, odd degree $d$ and generic weight $\alpha$ with full flags at $n$ points is given by
\begin{align*}
    C(\mathcal{M}^\alpha_{\mathrm{Higgs}}(\xi))_{\mathbb{Z}} \cong &C(\mathcal{M}(\xi))_{\mathbb{Z}} \otimes C(\mathbb{P}^1)^{\otimes n}_{\mathbb{Z}}\\ &\oplus \bigoplus_{{0 \leq l \leq n} \atop {\tfrac{l+1-n}{2} \leq j \leq g-1}} C(\widetilde{\mathrm{Sym}}^{2j+n-l-1}(X))_{\mathbb{Z}}(3g-2j+l-2)^{\oplus \binom{n}{l}},
\end{align*}
where $\widetilde{\mathrm{Sym}}^i(X) \to \mathrm{Sym}^i(X)$ is the \'etale cover of degree $2^{2g}$ induced by the multiplication by-$2$ map on $\mathrm{Jac}(X)$ and $\mathcal{M}(\xi) = \mathcal{M}^\xi(2,d)$ is the moduli space of semistable vector bundles with fixed determinant $\xi$. Since the cohomology of $\mathcal{M}^\alpha_{\mathrm{Higgs}}(\xi)$ is pure, from the above theorems \ref{fix2} and \ref{fix3}, we get a description of the Voevodsky (or, Chow) motives of the moduli space of parabolic connections with fixed determinant.

\section*{Acknowledgement}
We sincerely appreciate the anonymous reviewer whose insights and recommendations helped us to make this manuscript clearer and better. This work was supported by the Institute for Basic Science (IBS-R003-D1).

\end{document}